\documentclass{amsart}
\usepackage{amsmath,amssymb,amsthm}
\usepackage{graphicx}

\usepackage[colorlinks=true,linkcolor=blue,citecolor=blue,pdfpagelabels=false]{hyperref}

\theoremstyle{plain}
  \newtheorem{theorem}{Theorem}[section]
  \newtheorem{lemma}[theorem]{Lemma}

\theoremstyle{definition}
  \newtheorem{example}[theorem]{Example}
  \newtheorem{remark}[theorem]{Remark}

\newcommand{\assign}{:=}
\newcommand{\dueto}[1]{\textup{\textbf{(#1) }}}
\newcommand{\emdash}{---}
\newcommand{\mathd}{\mathrm{d}}
\newcommand{\nin}{\not\in}
\newcommand{\tmem}[1]{{\em #1\/}}

\newcommand{\mathbbm}[1]{\mathbb{#1}}

\title{On gamma quotients and infinite products}

\author{Marc Chamberland}
\address{Department of Mathematics and Statistics,
Grinnell College,
1116 Eighth Avenue,
Grinnell, IA 50112,
United States}
\email{chamberl@math.grinnell.edu}

\author{Armin Straub}
\address{Department of Mathematics,
University of Illinois at Urbana-Champaign,
1409 W. Green St,
Urbana, IL 6180,
United States}
\curraddr{Max-Planck-Institut f\"ur Mathematik,
Vivatsgasse 7,
53111 Bonn,
Germany}
\email{astraub@illinois.edu}

\begin{document}

\begin{abstract}
  Convergent infinite products, indexed by all natural numbers, in which each
  factor is a rational function of the index, can always be evaluated in terms
  of finite products of gamma functions. This goes back to Euler. A purpose of
  this note is to demonstrate the usefulness of this fact through a number of
  diverse applications involving multiplicative partitions, entries in
  Ramanujan's notebooks, the Chowla--Selberg formula, and the Thue--Morse
  sequence. In addition, we propose a numerical method for efficiently
  evaluating more general infinite series such as the slowly convergent
  Kepler--Bouwkamp constant.
\end{abstract}


\keywords{infinite products; gamma function; Kepler--Bouwkamp constant; Chowla--Selberg formula; Thue--Morse sequence}


\date{May 28, 2013}

\maketitle

\section{Introduction}

Recall that an infinite product $\prod_{k = 1}^{\infty} a (k)$ is said to
converge if the sequence of its partial products converges to a nonzero limit.
In this note we are especially interested in the case when $a (k)$ is a
rational function of $k$. Assuming that the infinite product converges, $a
(k)$ is then necessarily of the form
\begin{equation}
  a (k) = \frac{(k + \alpha_1) \cdots (k + \alpha_n)}{(k + \beta_1) \cdots (k
  + \beta_n)} \label{eq:rat}
\end{equation}
with $\alpha_1, \ldots, \alpha_n$ and $\beta_1, \ldots, \beta_n$ complex
numbers, none of which are negative integers, such that $\alpha_1 + \ldots +
\alpha_n = \beta_1 + \ldots + \beta_n$. To see that this is the case, note
first that clearly $a (k) \rightarrow 1$ as $k \rightarrow \infty$ so that $a
(k)$ can be factored into linear terms as on the right-hand side of
(\ref{eq:rat}). On the other hand, if $a (k) = 1 + c k^{- 1} + O (k^{- 2})$ as
$k \rightarrow \infty$ then convergence of the infinite product (and
divergence of the harmonic series) forces $c = \alpha_1 + \ldots + \alpha_n -
\beta_1 - \ldots - \beta_n = 0$.

These infinite products always have a finite-term evaluation in terms of
Euler's gamma function {\cite[Sec. 12.13]{ww}}.

\begin{theorem}
  \label{thm:prodrat}Let $n \geqslant 1$ be an integer, and let $\alpha_1,
  \ldots, \alpha_n$ and $\beta_1, \ldots, \beta_n$ be nonzero complex numbers,
  none of which are negative integers. If $\alpha_1 + \ldots + \alpha_n =
  \beta_1 + \ldots + \beta_n$, then
  \begin{equation}
    \prod_{k \geqslant 0} \frac{(k + \alpha_1) \cdots (k + \alpha_n)}{(k +
    \beta_1) \cdots (k + \beta_n)} = \frac{\Gamma (\beta_1) \cdots \Gamma
    (\beta_n)}{\Gamma (\alpha_1) \cdots \Gamma (\alpha_n)} .
    \label{eq:prodrat}
  \end{equation}
  Otherwise, the infinite product in (\ref{eq:prodrat}) diverges.
\end{theorem}

This result is a simple consequence of Euler's infinite product definition
(\ref{eq:gamma-euler}) of the gamma function, see the beginning of Section
\ref{sec:basic}. It is, however, scarcely stated explicitly in the literature.
For instance, while the table {\cite{hansen}} contains several pages of
special cases of (\ref{eq:prodrat}), some of which are rather generic in
nature, it does not list (\ref{eq:prodrat}) or an equivalent version thereof.
An incidental objective of this note is therefore to advertise
(\ref{eq:prodrat}) and to illustrate its usefulness during the course of the
applications given herein.

This note was motivated by a result, discussed in Section \ref{sec:appl},
which recently appeared in {\cite{chamberland-multpart}} as part of a study of
multiplicative partitions. In Section \ref{sec:appl} we also apply Theorem
\ref{thm:prodrat} to two entries in Ramanujan's (lost) notebook
{\cite{rln-4}}.

The first novel contribution of this note may be found in Section
\ref{sec:num}, where we propose an approach to the numerical evaluation of
certain general, not necessarily rational, infinite products, which is based
upon Theorem \ref{thm:prodrat} and Pad\'e approximation. We illustrate this
approach by applying it to the Kepler--Bouwkamp constant, defined as the
infinite product $\prod_{k = 3}^{\infty} \cos \left( \pi / k \right)$. Due to
its infamously slow convergence, various procedures for its numerical
evaluation have been discussed in the literature {\cite{bouwkamp-prod}},
{\cite{grimstone-prod}}, {\cite{stephens-prod}}. The present approach has the
advantage that it does not rely on developing alternative, more rapidly
convergent, expressions for the Kepler--Bouwkamp constant.

In Section \ref{sec:shortgamma} we discuss properties of short gamma quotients
at rational arguments. In particular, we offer an alternative proof of a
result established in {\cite{sandor-gamma}} and {\cite{martin-gamma}}. In the
light of our proof, this result may be interpreted as a (much simpler) version
of the Chowla--Selberg formula {\cite{sc67}} in the case of principal
characters.

Finally, in Section \ref{sec:tm}, consideration of an infinite product defined
in terms of the Thue--Morse sequence naturally leads us to a curious open
problem posed by Shallit.

\section{Proof and basic examples}\label{sec:basic}

We commence with supplying a proof of Theorem \ref{thm:prodrat} and giving a
number of basic examples.

\begin{proof}[Proof of Theorem \ref{thm:prodrat}]
  Euler's definition gives the gamma function as
  \begin{equation}
    \Gamma (z) = \lim_{m \rightarrow \infty} \frac{m^z m!}{z (z + 1) \cdots (z
    + m)}, \label{eq:gamma-euler}
  \end{equation}
  which is valid for all $z \in \mathbbm{C}$ except for negative integers $z$.
  Thus,
  \begin{eqnarray*}
    \prod_{j = 1}^n \frac{\Gamma \left( \beta_j \right)}{\Gamma \left(
    \alpha_j \right)} & = & \lim_{m \rightarrow \infty} \prod_{j = 1}^n
    m^{\beta_j - \alpha_j} \prod_{k = 0}^m \frac{\alpha_j + k}{\beta_j + k}\\
    & = & \lim_{m \rightarrow \infty} \prod_{k = 0}^m \prod_{j = 1}^n
    \frac{\alpha_j + k}{\beta_j + k},
  \end{eqnarray*}
  where, for the second equality, we make use of the fact that the sum of the
  $\alpha_j$ is the same as the sum of the $\beta_j$.
\end{proof}

In a similarly straight-forward manner, see {\cite[Sec. 12.13]{ww}}, an
alternative proof of Theorem \ref{thm:prodrat} follows from the Weierstrassian
infinite product
\[ \frac{1}{\Gamma \left( 1 + z \right)} = e^{\gamma z} \prod_{k \geqslant 1}
   \left( 1 + \frac{z}{k} \right) e^{- z / k} . \]
\begin{example}
  \label{eg:sineprod}In the case $\alpha_1 = z$, $\alpha_2 = - z$, $\beta_1 =
  \beta_2 = 0$, Theorem \ref{thm:prodrat} yields the famous
  \begin{equation}
    \prod_{k \geqslant 1} \left( 1 - \frac{z^2}{k^2} \right) = \frac{1}{\Gamma
    \left( 1 - z \right) \Gamma \left( 1 + z \right)} = \frac{\sin \left( \pi
    z \right)}{\pi z} . \label{eq:sine-euler}
  \end{equation}
  Similarly, one finds, for non-integral $z$, the slightly less well-known
  \[ \prod_{k \geqslant 1} \frac{\left( k - z \right) \left( k + z - 1
     \right)}{\left( k - 1 / 2 \right)^2} = \frac{\Gamma \left( 1 / 2
     \right)^2}{\Gamma \left( 1 - z \right) \Gamma \left( z \right)} = \sin
     \left( \pi z \right) . \]
  Both representations clearly reflect the reflection formula,
  \begin{equation}
    \Gamma (z) \Gamma (1 - z) = \frac{\pi}{\sin (\pi z)}, \label{eq:gammasin}
  \end{equation}
  for the gamma function.
\end{example}

\begin{example}
  \label{eg:wallis}Theorem \ref{thm:prodrat} also gives an immediate proof of
  Wallis' product
  \[ \frac{2 \cdot 2}{1 \cdot 3} \cdot \frac{4 \cdot 4}{3 \cdot 5} \cdot
     \frac{6 \cdot 6}{5 \cdot 7} \cdots = \prod_{k \geqslant 0} \frac{(2 k +
     2) (2 k + 2)}{(2 k + 1) (2 k + 3)} = \frac{\Gamma (1 / 2) \Gamma (3 /
     2)}{\Gamma (1) \Gamma (1)} = \frac{\pi}{2} . \]
  Alternatively, this evaluation can be seen as a corollary to
  (\ref{eq:sine-euler}). Several generalizations of Wallis' product, similarly
  based on Theorem \ref{thm:prodrat}, are discussed in the recent
  {\cite{sondow-wallis}} and {\cite{bhr-wallis}}.
\end{example}

\begin{example}
  As another illustration of Theorem \ref{thm:prodrat}, we evaluate
  \[ \prod_{k \geqslant 1} \left( 1 - \frac{(- 1)^k}{(2 k + 1)^3} \right) . \]
  This is Entry 89.6.12 in {\cite{hansen}} where it is incorrectly listed with
  value $\pi^3 / 32$, but is corrected in an erratum. Let $\rho = (3 + i
  \sqrt{3}) / 8$. The correct, though more involved, value is
  \begin{eqnarray*}
    \prod_{k \geqslant 1} \left( 1 - \frac{(- 1)^k}{(2 k + 1)^3} \right) & = &
    \prod_{k \geqslant 1} \left( 1 - \frac{1}{(4 k + 1)^3} \right) \prod_{k
    \geqslant 1} \left( 1 + \frac{1}{(4 k - 1)^3} \right)\\
    & = & \prod_{k \geqslant 1} \frac{k (k + \rho) (k + \bar{\rho})}{(k + 1 /
    4)^3}  \prod_{k \geqslant 1} \frac{k (k - \rho) (k - \bar{\rho})}{(k - 1 /
    4)^3}\\
    & = & \frac{\Gamma (1 + 1 / 4)^3}{\Gamma (1 + \rho) \Gamma (1 +
    \bar{\rho})}  \frac{\Gamma (1 - 1 / 4)^3}{\Gamma (1 - \rho) \Gamma (1 -
    \bar{\rho})}\\
    & = & \left( \frac{\pi / 4}{\sin (\pi / 4)} \right)^3  \frac{\sin (\pi
    \rho)}{\pi \rho}  \frac{\sin (\pi \bar{\rho})}{\pi \bar{\rho}}\\
    & = & \frac{\pi}{12} \left[ 1 + \sqrt{2} \cosh \left( \frac{\sqrt{3}}{4}
    \pi \right) \right] 
  \end{eqnarray*}
  For the last equality we used that $| \sin (x + i y) |^2 = (\cosh (2 y) -
  \cos (2 x)) / 2$.
\end{example}

\begin{example}
  This example briefly indicates that Theorem \ref{thm:prodrat} also applies
  to infinite products with individual terms removed, for instance, for
  convergence. Let $\xi_n = e^{2 \pi i / n}$ and $z \nin \{0, 1, 2, \ldots\}$.
  The evaluation
  \[ \prod_{k \geqslant 0} \frac{k^n - z^n}{k^n + z^n} = \prod_{j = 1}^{2 n}
     \Gamma (z \xi_{2 n}^j)^{(- 1)^{j + 1}} \]
  follows from Theorem \ref{thm:prodrat} as well. Let $m$ be a nonnegative
  integer. Since the residue of the gamma function $\Gamma (z)$ at $z = - m$
  is $(- 1)^m / m!$, we have, as $z \rightarrow m$,
  \[ (m^n - z^n) \Gamma (- z) \rightarrow \frac{(- 1)^m}{m!} n m^{n - 1}, \]
  and hence
  \[ \prod_{k \geqslant 0, k \neq m} \frac{k^n - m^n}{k^n + m^n} = (- 1)^m m!
     \frac{2 m}{n} \prod_{j = 1}^{2 n - 1} \Gamma (- m \xi_{2 n}^j)^{(- 1)^{j
     + 1}} . \]
  This example is discussed in much more detail in {\cite[Section 1.2]{bbg}},
  where it is also noted that the gamma functions can be replaced by
  trigonometric functions when $n$ is even.
\end{example}

\section{Further applications}\label{sec:appl}

\subsection{Multiplicative partitions}

The original motivation for this note was the following result, which recently
appeared in {\cite{chamberland-multpart}} as part of a study of multiplicative
partitions.

\begin{theorem}
  {\dueto{{\cite[Theorem 4.2]{chamberland-multpart}}}}\label{thm:marc}For
  integers $n \geqslant 2$,
  \[ \prod_{k \geqslant 2} \frac{1}{1 - k^{- n}} = \prod_{j = 1}^{n - 1}
     \Gamma \left( 2 - \xi_n^j \right) = n \prod_{j = 1}^{n - 1} \Gamma \left(
     1 - \xi_n^j \right) \]
  where $\xi_n = e^{2 \pi i / n}$.
\end{theorem}

\begin{proof}
  Apply Theorem \ref{thm:prodrat} to the rational function $\frac{(k +
  2)^n}{(k + 2)^n - 1}$ to obtain the first equality. For the second part,
  recall that $\Gamma (x + 1) = x \Gamma (x)$.
\end{proof}

To put Theorem \ref{thm:marc} into context, let $a_n$ denote the number of
multiplicative partitions of the natural number $n$. For instance, $a_{18} =
4$ because $18 = 2 \cdot 9 = 2 \cdot 3 \cdot 3 = 3 \cdot 6$. In analogy with
Euler's infinite product formula for the zeta function, the Dirichlet
generating series for the $a_n$ is the product
\[ \sum_{n = 1}^{\infty} \frac{a_n}{n^s} = \prod_{k \geqslant 2} (1 + k^{- s}
   + k^{- 2 s} + \ldots) = \prod_{k \geqslant 2} \frac{1}{1 - k^{- s}} . \]
The values in Theorem \ref{thm:marc} are values of this Dirichlet series at
positive integers and, as such, analogous to the zeta values $\zeta (n)$.

\subsection{A product considered by Ramanujan}

In {\cite{ramanujan-prod15}}, see also {\cite[Chapter 16]{rln-4}}, Ramanujan
considers the product
\begin{equation}
  \phi (\alpha, \beta) = \prod_{n = 1}^{\infty} \left\{ 1 + \left(
  \frac{\alpha + \beta}{n + \alpha} \right)^3 \right\}, \label{eq:rama:phi}
\end{equation}
and shows that $\phi (\alpha, \beta)$ can be expressed in ``finite terms''
when the difference $\alpha - \beta$ is an integer. If one includes values of
the gamma function in the notion of finite term, then the product
(\ref{eq:rama:phi}) can always be evaluated in finite terms. Indeed, using the
factorization
\[ 1 + \left( \frac{\alpha + \beta}{n + \alpha} \right)^3 = \frac{(n + 2
   \alpha + \beta) \left( n + \frac{\alpha - \beta + (\alpha + \beta) \sqrt{-
   3}}{2} \right) \left( n + \frac{\alpha - \beta - (\alpha + \beta) \sqrt{-
   3}}{2} \right)}{(n + \alpha)^3}, \]
we have, by Theorem \ref{thm:prodrat},
\begin{equation}
  \phi (\alpha, \beta) = \frac{\Gamma \left( 1 + \alpha \right)^3}{\Gamma (1 +
  2 \alpha + \beta) \Gamma \left( 1 + \frac{\alpha - \beta + i (\alpha +
  \beta) \sqrt{3}}{2} \right) \Gamma \left( 1 + \frac{\alpha - \beta - i
  (\alpha + \beta) \sqrt{3}}{2} \right)} . \label{eq:rama:phigamma}
\end{equation}
We observe that, if $\alpha - \beta$ is an integer, then the product
\[ \Gamma \left( 1 + \frac{\alpha - \beta + i (\alpha + \beta) \sqrt{3}}{2}
   \right) \Gamma \left( 1 + \frac{\alpha - \beta - i (\alpha + \beta)
   \sqrt{3}}{2} \right) \]
can be simplified using the reflection formula (\ref{eq:gammasin}). As noted
by Ramanujan, one arrives at an evaluation of $\phi (\alpha, \beta)$ in terms
of hyperbolic functions only.

\begin{example}
  \label{eg:phi-aa}In the case $\alpha = \beta$, we have from
  (\ref{eq:rama:phigamma})
  \begin{eqnarray}
    \phi (\alpha, \alpha) & = & \frac{\Gamma (1 + \alpha)^3}{\Gamma (1 + 3
    \alpha) \Gamma \left( 1 + i \alpha \sqrt{3} \right) \Gamma \left( 1 - i
    \alpha \sqrt{3} \right)} \nonumber\\
    & = & \frac{\Gamma (1 + \alpha)^3}{\Gamma (1 + 3 \alpha)}  \frac{\sinh
    (\pi \alpha \sqrt{3})}{\pi \alpha \sqrt{3}}, 
  \end{eqnarray}
  as in {\cite[equation (6)]{ramanujan-prod15}} and {\cite[equation
  (16.3.1)]{rln-4}}.
\end{example}

\subsection{Another product considered by Ramanujan}

In the next example we employ Theorem \ref{thm:prodrat} to rewrite an
integral, considered by Ramanujan in his lost notebook {\cite{rln-4}}. By
doing so, we obtain a Mellin--Barnes integral which provides further context
for the integral evaluation and leads to natural generalizations.

In {\cite[Entry 4.9.1]{rln-4}} the following integral due to Ramanujan is
recorded:
\begin{eqnarray}
  &  & \int_0^{\infty} \left( \frac{1 + x^2 / b^2}{1 + x^2 / a^2} \right)
  \left( \frac{1 + x^2 / (b + 1)^2}{1 + x^2 / (a + 1)^2} \right) \left(
  \frac{1 + x^2 / (b + 2)^2}{1 + x^2 / (a + 2)^2} \right) \cdots \mathd x
  \nonumber\\
  & = & \frac{\sqrt{\pi}}{2}  \frac{\Gamma (a + 1 / 2) \Gamma (b) \Gamma (b -
  a - 1 / 2)}{\Gamma (a) \Gamma (b - 1 / 2) \Gamma (b - a)}, 
  \label{eq:rama491}
\end{eqnarray}
where $0 < a < b - \tfrac{1}{2}$. Using the factorization
\[ \frac{1 + x^2 / (b + k)^2}{1 + x^2 / (a + k)^2} = \frac{(k + a)^2 (k + b +
   i x) (k + b - i x)}{(k + b)^2 (k + a + i x) (k + a - i x)}, \]
as well as Theorem \ref{thm:prodrat}, the integrand can be rewritten as
\[ \prod_{k = 0}^{\infty} \frac{1 + x^2 / (b + k)^2}{1 + x^2 / (a + k)^2} =
   \frac{\Gamma (b)^2 \Gamma (a + i x) \Gamma (a - i x)}{\Gamma (a)^2 \Gamma
   (b + i x) \Gamma (b - i x)} . \]
Equation \ref{eq:rama491} is therefore equivalent to the Mellin--Barnes
integral
\begin{equation}
  \frac{1}{2 \pi i} \int_{- i \infty}^{i \infty} \frac{\Gamma (a + s) \Gamma
  (a - s)}{\Gamma (b + s) \Gamma (b - s)} \mathd s = \frac{1}{2 \sqrt{\pi}} 
  \frac{\Gamma (a)}{\Gamma (b)}  \frac{\Gamma (a + 1 / 2) \Gamma (b - a - 1 /
  2)}{\Gamma (b - 1 / 2) \Gamma (b - a)} . \label{eq:rama-491b}
\end{equation}
We note that, using the duplication formula
\begin{equation}
  \Gamma (2 z) = \frac{2^{2 z - 1}}{\sqrt{\pi}} \Gamma (z) \Gamma \left( z +
  \tfrac{1}{2} \right), \label{eq:gammadup}
\end{equation}
the right-hand side of (\ref{eq:rama-491b}) can be simplified to yield
\[ \frac{1}{2 \pi i} \int_{- i \infty}^{i \infty} \frac{\Gamma (a + s) \Gamma
   (a - s)}{\Gamma (b + s) \Gamma (b - s)} \mathd s = \frac{\Gamma (2 a)
   \Gamma (2 b - 2 a - 1)}{\Gamma (b - a)^2 \Gamma (2 b - 1)} . \]
In this final, somewhat more canonical, form it is straight-forward to find
natural generalizations in the literature, such as
\[ \frac{1}{2 \pi i} \int_{- i \infty}^{i \infty} \frac{\Gamma (a + s) \Gamma
   (c - s)}{\Gamma (b + s) \Gamma (d - s)} \mathd s = \frac{\Gamma (a + c)
   \Gamma (b + d - a - c - 1)}{\Gamma (b - a) \Gamma (d - c) \Gamma (b + d -
   1)}, \]
which is proved in {\cite{pk-mellinbarnes}} using Parseval's formula.

\section{Numerical applications}\label{sec:num}

\subsection{Numerical evaluations of products}\label{sec:kb}

In this section, we consider general infinite products $\prod_{k = 1}^{\infty}
a (k)$, where $a (k)$ is not necessarily a rational function. The goal is to
present a simple yet efficient way to obtain accurate numerical evaluations of
such infinite products for certain $a (k)$, even when the original product
converges very slowly. The approach is based on approximating $a (k)$ by a
rational function and using Theorem \ref{thm:prodrat} to express the result as
a finite product of gamma functions.

We illustrate this approach for the Kepler--Bouwkamp constant {\cite[Sec.
6.3]{finch-constants}}
\begin{equation}
  \prod_{k = 3}^{\infty} \cos \left( \frac{\pi}{k} \right) =
  0.1149420448532962 \ldots \label{eq:kb}
\end{equation}
This constant is motivated by a geometric construction. Start with a circle of
unit radius and inscribe an equilateral triangle, inscribe the triangle with
another circle which is then inscribed with a square, inscribe the square with
yet another circle which is inscribed with a regular pentagon, and so on as in
Figure \ref{fig:kb}. The radii of the inscribing circles then approach a limit
which is the Kepler--Bouwkamp constant (\ref{eq:kb}). For further references
and the history of this constant, including various approaches to its
numerical computation, we refer to {\cite{bouwkamp-prod}}, {\cite[Sec.
6.3]{finch-constants}}, {\cite{grimstone-prod}} and {\cite{stephens-prod}}.

The product (\ref{eq:kb}), however, converges rather slowly. For instance,
truncating the product after $10^4$ terms only results in four correct digits.
The usual approach to computing the Kepler--Bouwkamp constant {\emdash} taken,
for instance, in {\cite{bouwkamp-prod}} and {\cite{stephens-prod}} {\emdash}
is to first develop more rapidly converging expressions for (\ref{eq:kb}). On
the other hand, we will demonstrate how one can use (\ref{eq:kb}) to evaluate
the Kepler--Bouwkamp constant in a completely automated way to, say, $100$
digits in a matter of a few seconds.

\begin{figure}[h]
  \begin{center}
    \includegraphics[width=5cm]{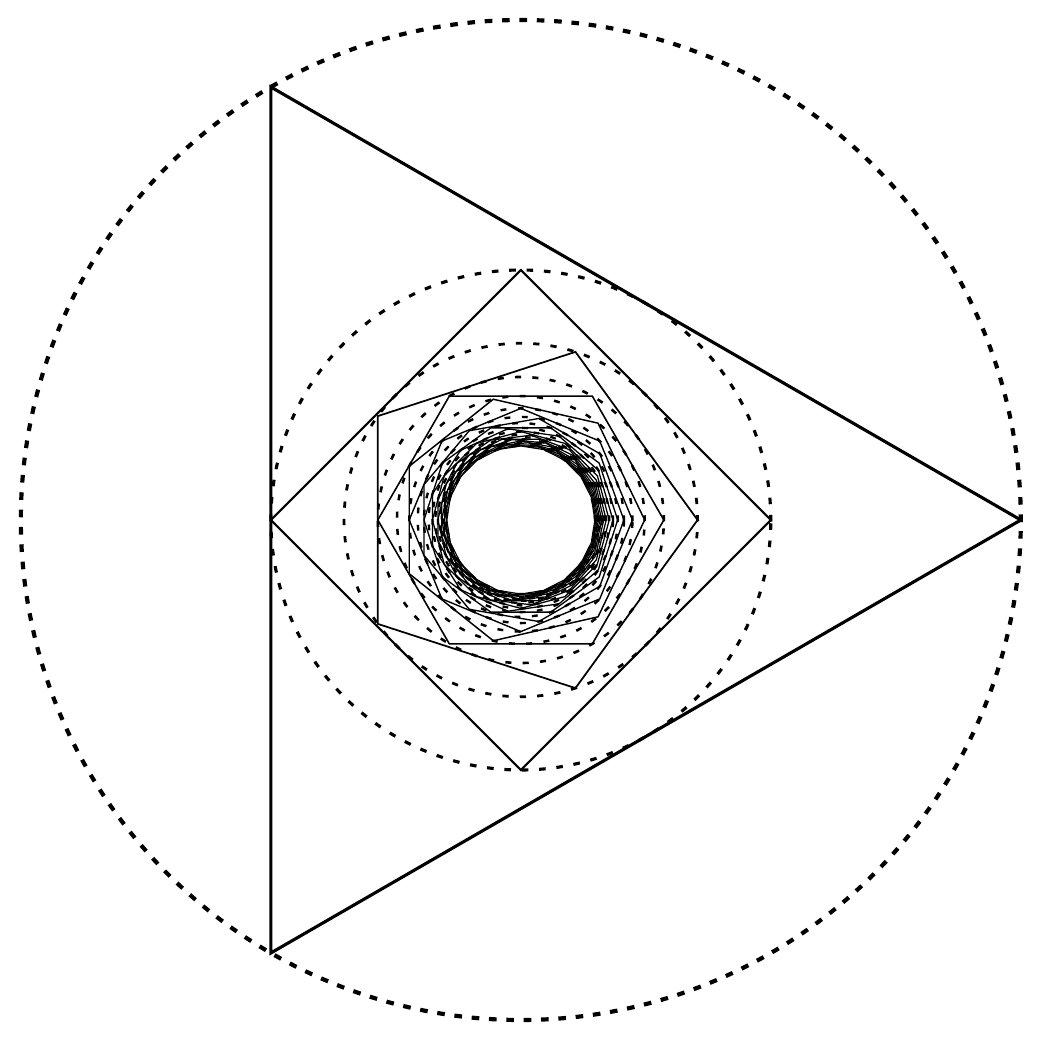}
    \caption{\label{fig:kb}Kepler--Bouwkamp constant as ratio of inner circle
    and outer circle}
  \end{center}
\end{figure}

A natural choice for approximating a function $f (x)$, such as $\cos (\pi x)$,
by a rational function is to use a Pad\'e approximant {\cite[Chapter
4]{handbook-contfrac}}. The Pad\'e approximant of order $[m, n]$ is the
rational function of numerator degree $m$ and denominator degree $n$, whose
Maclaurin series agrees with the one of $f (x)$ to order $m + n$ (the highest
possible order). For our purposes we are interested only in Pad\'e
approximants of order $[n, n]$. For instance, the $[2, 2]$ Pad\'e approximant
of $\cos (x)$ is
\[ r_2 (x) = \frac{12 - 5 x^2}{12 + x^2} = \cos (x) + O (x^6) . \]
We can now approximate the product (\ref{eq:kb}) by the corresponding product
of rational functions, which, using Theorem \ref{thm:prodrat}, evaluates to a
finite product of gamma functions:
\[ \prod_{k = 3}^{\infty} r_2 \left( \tfrac{\pi}{k} \right) = \prod_{k =
   3}^{\infty} \frac{12 k^2 - 5 \pi^2}{12 k^2 + \pi^2} = \frac{\Gamma \left( 3
   - \tfrac{i}{6} \sqrt{3} \pi \right) \Gamma \left( 3 + \tfrac{i}{6} \sqrt{3}
   \pi \right)}{\Gamma \left( 3 - \tfrac{1}{6} \sqrt{15} \pi \right) \Gamma
   \left( 3 + \tfrac{1}{6} \sqrt{15} \pi \right)}, \]
This approximation agrees with (\ref{eq:kb}) to three decimal digits. More
accurate approximations can be obtained by using the Pad\'e approximant only
for $k \geqslant N$, thus approximating (\ref{eq:kb}) with
\[ \left[ \prod_{k = 3}^{N - 1} \cos \left( \tfrac{\pi}{k} \right) \right]
   \left[ \prod_{k = N}^{\infty} r_2 \left( \tfrac{\pi}{k} \right) \right] =
   \left[ \prod_{k = 3}^{N - 1} \cos \left( \tfrac{\pi}{k} \right) \right]
   \frac{\Gamma \left( N - \tfrac{i}{6} \sqrt{3} \pi \right) \Gamma \left( N +
   \tfrac{i}{6} \sqrt{3} \pi \right)}{\Gamma \left( N - \tfrac{1}{6} \sqrt{15}
   \pi \right) \Gamma \left( N + \tfrac{1}{6} \sqrt{15} \pi \right)} . \]
For $N = 10$ this results in $6$ correct digits, and $11$ correct digits for
$N = 100$. Of course, more accurate approximations are obtained if the order
$n$ of the Pad\'e approximant is increased. Table \ref{tbl:kbappr} shows the
number of correct decimal digits that one obtains for various modest choices
of $n$ and $N$ (instead of truncating the entries of Table \ref{tbl:kbappr} to
integers, we adopt the convention that two numbers $A$ and $B$ agree to $-
\log_{10} |A - B|$ decimal digits). None of the computations took more than 5
seconds on a usual laptop using Mathematica 7 and without optimizing the
computation.

\begin{table}[h]
  \begin{center}
  \begin{tabular}{|c|c|c|c|c|c|c|}
    \hline
    $N$ & $3$ & $4$ & $5$ & $10$ & $100$ & $1000$\\
    \hline
    $n = 2$ & $3.19$ & $4.00$ & $4.57$ & $6.22$ & $11.3$ & $16.3$\\
    \hline
    $n = 4$ & $6.87$ & $8.22$ & $9.21$ & $12.1$ & $21.3$ & $30.3$\\
    \hline
    $n = 6$ & $11.2$ & $13.1$ & $14.5$ & $18.7$ & $31.9$ & $45.0$\\
    \hline
    $n = 8$ & $16.1$ & $18.5$ & $20.3$ & $25.7$ & $43.0$ & $60.1$\\
    \hline
    $n = 10$ & $21.4$ & $24.3$ & $26.5$ & $33.1$ & $54.5$ & $75.5$\\
    \hline
    $n = 12$ & $27.0$ & $30.4$ & $33.0$ & $40.8$ & $66.2$ & $91.3$\\
    \hline
    $n = 14$ & $32.9$ & $36.8$ & $39.7$ & $48.8$ & $78.3$ & $107.$\\
    \hline
    $n = 16$ & $39.0$ & $43.4$ & $46.7$ & $57.0$ & $90.5$ & $124.$\\
    \hline
  \end{tabular}
  \end{center}
  \caption{\label{tbl:kbappr}Number of correct digits for various
  approximations to (\ref{eq:kb})}
\end{table}

We note that this method for approximating an infinite product $\prod_{k =
1}^{\infty} a (k)$ only relies on computing a Pad\'e approximant for the
factor $a (k)$, numerically finding the zeros and poles of this rational
function in order to apply Theorem \ref{thm:prodrat}, and numerically
evaluating the resulting finite product of gamma functions. All these
operations are efficiently and easily available in any computer algebra
system. This makes the present approach a rather versatile tool in numerically
evaluating a number of slowly converging infinite products.

\begin{remark}
  Because of their special relevance in number theory it is natural to ask if
  the above numerical procedure can be applied to products indexed by primes,
  such as Artin's constant
  \[ \prod_{\text{$p$ prime}} \left( 1 - \frac{1}{p \left( p - 1 \right)}
     \right) \approx 0.373956. \]
  However, no immediate transfer appears available. Fortunately, a very
  efficient method for numerically evaluating such products has been developed
  in {\cite{moree-prod}}. We cannot resist to remark that, for Artin's
  constant, the corresponding product over all integers is given by
  \[ \prod_{n \geqslant 2} \left( 1 - \frac{1}{n \left( n - 1 \right)} \right)
     = - \frac{1}{\pi} \cos \left( \frac{\sqrt{5}}{2} \pi \right) \approx
     0.296675. \]
\end{remark}

\subsection{Numerical evaluations of sums}

Note that series $\sum_{k = 0}^{\infty} a_k$ are related to products via the
obvious
\begin{equation}
  \sum_{k = 0}^{\infty} a_k = \log \left( \prod_{k = 0}^{\infty} \exp (a_k)
  \right) . \label{eq:sumprod}
\end{equation}
The approach for numerically evaluating infinite products may thus also be
applied to certain series.

In case of the exponential function, explicit formulas exist {\cite[Example
4.2.2]{handbook-contfrac}} for the Pad\'e approximants of any order. In
particular, letting
\[ f_n (x) = \sum_{j = 0}^n \frac{(2 n - j) ! n!}{(2 n) ! j! (n - j) !} x^j =
   \sum_{j = 0}^n \frac{\binom{n}{j}}{\binom{2 n}{j}}  \frac{x^j}{j!}, \]
the $[n, n]$ Pad\'e approximation of $\exp (x)$ is given by
\[ \frac{f_n (x)}{f_n (- x)} . \]
In fact, as explained in the lovely article {\cite{cohn-e}}, these
approximations to the exponential function are already implicit in Hermite's
1873 paper on the transcendence of $e$, and thus predate the systematic study
of Pad\'e approximations by Hermite's eponymous student Pad\'e in his 1892
thesis.

For instance, in the case of the Riemann zeta function we obtain the
approximations
\begin{equation}
  \zeta (m) = \sum_{k = 1}^{\infty} \frac{1}{k^m} \approx \log \left( \prod_{k
  = 1}^{\infty} \frac{f_n (k^{- m})}{f_n (- k^{- m})} \right) = : \zeta_n (m)
  \label{eq:zetanm}
\end{equation}
by replacing the exponential function by its $[n, n]$ Pad\'e approximation in
(\ref{eq:sumprod}). Table \ref{tbl:zetaappr} gives some indication on the
quality of this approximation (again, we say that two numbers $A$ and $B$
agree to $- \log_{10} |A - B|$ decimal digits). For instance, $\zeta_{10} (3)$
agrees with $\zeta (3)$ to $25$ digits. Indeed, a few computations quickly
suggest that, for any integer $m \geqslant 2$, $\zeta_{10} (m)$ agrees with
$\zeta (m)$ to $25$ decimal digits. Analogously, the data of Table
\ref{tbl:zetaappr} (to the precision given) applies for any integer $m
\geqslant 2$. This observation is, to some degree, explained next.

\begin{table}[h]
  \begin{center}
  \begin{tabular}{|c|c|c|c|c|c|c|c|c|c|}
    \hline
    $n$ & $2$ & $3$ & $4$ & $5$ & $6$ & $7$ & $8$ & $9$ & $10$\\
    \hline
    $m = 3$ & $2.83$ & $4.99$ & $7.39$ & $9.99$ & $12.8$ & $15.6$ & $18.7$ &
    $21.8$ & $25.0$\\
    \hline
  \end{tabular}
  \end{center}
  \caption{\label{tbl:zetaappr}Number of digits of $\zeta_n (3)$ that agree
  with $\zeta (3)$}
\end{table}

Clearly, we have $\zeta (m) \rightarrow 1$ as $m \rightarrow \infty$. On the
other hand, we see from (\ref{eq:zetanm}) that
\[ \lim_{m \rightarrow \infty} \zeta_n (m) = \log \left( \frac{f_n (1)}{f_n (-
   1)} \right) . \]
For instance,
\begin{equation}
  \lim_{m \rightarrow \infty} \zeta_3 (m) = \log \left( \frac{193}{71} \right)
  \approx 1.000010312 \label{eq:zetalim2},
\end{equation}
which elucidates the entry $4.99 \approx - \log_{10} |1 - \log (193 / 71) |$
for $n = 3$ in Table \ref{tbl:zetaappr}. The fact that the rational numbers
$f_n (1) / f_n (- 1)$, including $19 / 7$ for $n = 2$ and $193 / 71$ for $n =
3$, are convergents to the continued fraction of Euler's number $e$ is further
discussed in {\cite{cohn-e}}.

\section{Short gamma quotients}\label{sec:shortgamma}

In this section we discuss a few properties of gamma quotients whose arguments
are rational numbers. While no closed forms are known for $\Gamma (1 / n)$
when $n > 2$, it turns out that surprisingly short products of gamma functions
at rational arguments have simple evaluations {\cite{zucker-gamma}},
{\cite{sandor-gamma}}, {\cite{bz92}}, {\cite{vidunas-gamma}},
{\cite{martin-gamma}}, {\cite{nijenhuis-gamma}}. For instance, as proposed in
{\cite{glasser-11426}} and shown in {\cite{nijenhuis-gamma}},
\begin{equation}
  \Gamma \left( \frac{1}{14} \right) \Gamma \left( \frac{9}{14} \right) \Gamma
  \left( \frac{11}{14} \right) = 4 \pi^{3 / 2} . \label{eq:gammaprod-14}
\end{equation}
Let $\Phi (n)$ denote the set of integers between $1$ and $n$ which are
coprime to $n$. For example, $\Phi (12) = \{1, 5, 7, 11\}$. Let $\phi (n)$
represent the totient function, that is, the size of $\Phi (n)$. The general
result we present in Theorem \ref{thm:gammaprod-phi} below has been
established in {\cite{sandor-gamma}} and, independently, in
{\cite{martin-gamma}} (a nice generalization is developed in
{\cite{nijenhuis-gamma}}, see Remark \ref{rk:nij}). In both cases, the proof
is based on M\"obius inversion. We offer an alternative proof which rests upon
Lerch's identity {\cite[Theorem 1.3.4]{aar}}
\begin{equation}
  \log \Gamma (x) = \left. \frac{\partial}{\partial s} \zeta (s, x) \right|_{s
  = 0} + \frac{1}{2} \log (2 \pi), \label{eq:lerch}
\end{equation}
where $\zeta (s, x) = \sum_{n = 0}^{\infty} \frac{1}{(n + x)^s}$ is the
Hurwitz zeta function.

\begin{theorem}
  \label{thm:gammaprod-phi}If $n$ is not a prime power, then
  \begin{equation}
    \prod_{k \in \Phi (n)} \Gamma \left( \frac{k}{n} \right) = (2 \pi)^{\phi
    (n) / 2} . \label{eq:gammaprod-phi}
  \end{equation}
  If $n = p^m$ for a prime $p$, then (\ref{eq:gammaprod-phi}) holds with the
  right-hand side divided by $\sqrt{p}$.
\end{theorem}

\begin{proof}
  Let $\chi$ be the principal character modulo $n$ (that is, $\chi (k) = 1$ if
  $k$ is coprime to $n$, and $\chi (k) = 0$ otherwise), and let $p_1, \ldots,
  p_r$ be the distinct prime factors of $n$. Then the Dirichlet $L$-function
  associated to $\chi$ differs from the Riemann zeta function
  \[ \zeta (s) = \prod_{\text{$p$ prime}} \frac{1}{1 - p^{- s}} \]
  only in that the factors corresponding to the primes $p_1, \ldots, p_r$ are
  missing. In other words,
  \begin{equation}
    L (\chi, s) = \sum_{k = 1}^{\infty} \frac{\chi (k)}{k^s} = \zeta (s) 
    \prod_{j = 1}^r (1 - p_j^{- s}) . \label{eq:Lz}
  \end{equation}
  On the other hand,
  \[ n^s L (\chi, s) = \sum_{k = 1}^n \chi (k) \zeta \left( s, \frac{k}{n}
     \right) . \]
  Taking the derivative with respect to $s$ and applying Lerch's identity
  (\ref{eq:lerch}), we have
  \[ \sum_{k = 1}^n \chi (k) \log \Gamma \left( \frac{k}{n} \right) = \log (n)
     L (\chi, 0) + L' (\chi, 0) + \frac{1}{2} \phi (n) \log (2 \pi) . \]
  It follows from (\ref{eq:Lz}) that $L (\chi, 0) = 0$. Similarly, it follows
  that $L' (\chi, 0) = 0$ provided that $r > 1$. This proves equation
  (\ref{eq:gammaprod-phi}) unless $n$ is a prime power. Lastly, if $n = p^m$
  for some prime $p$, then
  \[ L' (\chi, 0) = \log (p) \zeta (0) = - \frac{1}{2} \log (p), \]
  which proves (\ref{eq:gammaprod-phi}) in the remaining case.
\end{proof}

\begin{example}
  \label{eg:gamma14}In the case $n = 14$, we find
  \[ \Gamma \left( \frac{1}{14} \right) \Gamma \left( \frac{3}{14} \right)
     \Gamma \left( \frac{5}{14} \right) \Gamma \left( \frac{9}{14} \right)
     \Gamma \left( \frac{11}{14} \right) \Gamma \left( \frac{13}{14} \right) =
     (2 \pi)^3 . \]
  In light of (\ref{eq:gammaprod-14}), it follows that
  \[ \Gamma \left( \frac{3}{14} \right) \Gamma \left( \frac{5}{14} \right)
     \Gamma \left( \frac{13}{14} \right) = 2 \pi^{3 / 2}, \]
  which, as in the case of (\ref{eq:gammaprod-14}), also follows from the
  results of {\cite{nijenhuis-gamma}}, which generalize Theorem
  \ref{thm:gammaprod-phi} and are outlined in Remark \ref{rk:nij}. Combining
  both products and using Theorem \ref{thm:prodrat}, one therefore obtains
  \[ \prod_{k \geqslant 0} \frac{(k + 3 / 14) (k + 5 / 14) (k + 13 / 14)}{(k +
     1 / 14) (k + 9 / 14) (k + 11 / 14)} = 2. \]
\end{example}

\begin{remark}
  \label{rk:nij}Note that $\Phi \left( n \right)$, the set of integers between
  $1$ and $n$ which are coprime to $n$, is a group with respect to
  multiplication modulo $n$. Assume that $n > 1$ is odd, and let $A$ be (a
  coset of) the cyclic subgroup of $\Phi \left( 2 n \right)$ generated by $n +
  2$. It is shown in {\cite{nijenhuis-gamma}} that
  \[ \prod_{k \in A} \Gamma \left( \frac{k}{2 n} \right) = 2^{b \left( A
     \right)} \pi^{\left| A \right| / 2}, \]
  with $\left| A \right|$ denoting the cardinality of $A$ and $b \left( A
  \right)$ the number of elements in $A$ which exceed $n$. In particular, this
  lets us construct lots of identities as in Example \ref{eg:gamma14}. To wit,
  for any $A$ as above,
  \[ \prod_{k \in A} \frac{\Gamma \left( \frac{k}{2 n} \right)}{\Gamma \left(
     1 - \frac{k}{2 n} \right)} = 2^{2 b \left( A \right) - \left| A \right|}
     . \]
  For instance, with $n = 31$, one choice is $A = \left\{ 1, 33, 35, 39, 47
  \right\}$ and we find
  \[ \prod_{k \geqslant 0} \frac{(k + 15 / 62) (k + 23 / 62) (k + 27 / 62)
     \left( k + 29 / 62 \right) \left( k + 61 / 62 \right)}{(k + 1 / 62) (k +
     33 / 62) (k + 35 / 62) \left( k + 39 / 62 \right) \left( k + 47 / 62
     \right)} = 8. \]
\end{remark}

\begin{remark}
  Equation (\ref{eq:gammaprod-phi}) may be used to produce the identity
  \begin{eqnarray*}
    \frac{\phi (n)}{2} \log (2 \pi) & = & \sum_{x \in \Phi (n)} \log \Gamma (x
    / n)\\
    & = & \sum_{x \in \Phi (n)} \log \Gamma (1 - x / n)\\
    & = & \sum_{x \in \Phi (n)} \left[ \gamma \frac{x}{n} + \sum_{k =
    2}^{\infty} \frac{\zeta (k)}{k} (x / n)^k \right]\\
    & = & \gamma \frac{\phi (n)}{2} + \sum_{k = 2}^{\infty} \frac{\zeta
    (k)}{k} \sum_{x \in \Phi (n)} (x / n)^k,
  \end{eqnarray*}
  where
  \[ \gamma = \sum_{k = 2}^{\infty} (- 1)^k \frac{\zeta (k)}{k} \]
  is Euler's constant. This may be re-arranged to obtain
  \begin{equation}
    \frac{\log (2 \pi) - \gamma}{2} = \sum_{k = 2}^{\infty} \frac{\zeta
    (k)}{k}  \frac{1}{\phi (n)} \sum_{x \in \Phi (n)} (x / n)^k,
    \label{eq:zetasumphi}
  \end{equation}
  again valid for $n$ which are not a prime power. Note that the expression on
  the left-hand side of (\ref{eq:zetasumphi}) has the interesting property
  that it is independent of $n$. The sums
  \[ \Psi_k (n) \assign \sum_{x \in \Phi (n)} x^k \]
  have been studied in {\cite{singh-powersums}}, where, among other results,
  it is shown that
  \[ \Psi_k (n) = \frac{n^{k + 1}}{k + 1} \sum_{m = 0}^{[k / 2]} \binom{k +
     1}{2 m} \frac{B_{2 m}}{n^{2 m}} \prod_{p|n} (1 - p^{2 m - 1}) . \]
\end{remark}

\begin{remark}
  We note that the proof of Theorem \ref{thm:gammaprod-phi} naturally extends
  to certain more general gamma quotients. For instance, consider the real
  character $\chi (n) = \left( \tfrac{- d}{n} \right)$ where $- d < 0$ is a
  negative fundamental discriminant (fundamental discriminants are squarefree
  integers congruent to $1$ modulo $4$, or multiples by $- 4$ or $- 8$ of such
  numbers) with associated class number $h = h (- d)$. Then
  \begin{equation}
    w \zeta (s) L (\chi, s) = \sum_{j = 1}^h Z_{Q_j} (s), \label{eq:sumzq}
  \end{equation}
  where $w$ is the number of roots of unity in $\mathbbm{Q} ( \sqrt{d})$,
  $Q_1, \ldots, Q_h$ are non-equivalent reduced binary quadratic forms with
  discriminant $- d$, and $Z_Q$ denotes the Epstein zeta function
  \[ Z_Q (s) = \sum_{n, m}' \frac{1}{Q (m, n)^s}, \]
  with the sum extending over all integers $n, m$ such that $(n, m) \neq (0,
  0)$. The {\tmem{Kronecker limit formula}} {\emdash} see {\cite{nw-epstein}}
  for a nice proof {\emdash} shows that
  \[ Z_Q (s) = \frac{a_{- 1}}{s - 1} + a_0 + a_1 (s - 1) + \ldots, \]
  with $a_{- 1} = 2 \pi / \sqrt{d}$ and
  \[ a_0 = \frac{4 \pi}{\sqrt{d}} (\gamma - \log (d^{1 / 4} \sqrt{2 y} | \eta
     (z) |^2)), \]
  where $z = x + i y$ is the solution to $Q (z, 1) = 0$ in the upper
  half-plane, and $\eta (\tau)$ is the Dedekind eta function. We now expand
  both sides of (\ref{eq:sumzq}) around $s = 1$ and equate the constant terms.
  For the right-hand side we employ the Kronecker limit formula, while for the
  left-hand side we proceed as in the proof of Theorem
  \ref{thm:gammaprod-phi}. Using $L (\chi, 0) = \frac{2 h}{w}$, this
  eventually yields
  \begin{equation}
    \prod_{m = 1}^d \Gamma \left( \frac{m}{d} \right)^{\left( \frac{- d}{m}
    \right)} = \left[ \prod_{j = 1}^h 4 \pi \sqrt{d} y_j | \eta (z_j) |^4
    \right]^{2 / w} . \label{eq:cs}
  \end{equation}
  This is the well-known Chowla--Selberg formula {\cite{sc67}}. We note that
  $w = 2$ unless $d = - 3$, in which case $w = 6$, or $d = - 4$, in which case
  $w = 4$. Generalizations of the Chowla--Selberg formula exist, for instance,
  to arbitrary negative discriminants and to genera of binary quadratic forms;
  we refer to {\cite{hkw-cs-genera}} and the references therein. Finally, we
  remark that, for fixed $d$, the individual eta values occurring in
  (\ref{eq:cs}) differ only by an algebraic factor.
\end{remark}

\begin{remark}
  Following {\cite{periods}}, a {\tmem{period}} is a number which is the value
  of an integral of an algebraic function over an algebraic domain. For
  example, as noted in the introduction to {\cite{periods}}, the numbers
  $\Gamma (p / q)^q$ are periods because they may be represented as beta
  integrals. More generally, let $\alpha_1, \ldots, \alpha_n$ be rational
  numbers. Then $\Gamma (\alpha_1) \cdots \Gamma (\alpha_n)$ is a period
  whenever $\Gamma (\alpha_1 + \ldots + \alpha_n)$ is a period. To see this,
  note that
  \[ \Gamma (\alpha_1) \cdots \Gamma (\alpha_n) = B (\alpha_1, \alpha_2) B
     (\alpha_1 + \alpha_2, \alpha_3) \cdots B (\alpha_1 + \ldots + \alpha_{n -
     1}, \alpha_n) \Gamma (\alpha_1 + \ldots + \alpha_n), \]
  where
  \[ B (\alpha, \beta) = \frac{\Gamma (\alpha) \Gamma (\beta)}{\Gamma (\alpha
     + \beta)} = \int_0^1 t^{\alpha - 1} (1 - t)^{\beta - 1} \mathd t \]
  is the beta function. Due to the integral representation its values for
  rational $\alpha, \beta$ are periods. It now suffices to observe that
  periods form a ring.
  
  A similar argument shows that gamma quotients arising from infinite products
  of rational functions with rational roots are always a quotient of two
  periods.
  
  Finally, we remark that, while the ring of periods is countable, it is an
  open problem, {\cite[Problem 3]{periods}}, to exhibit at least one number
  which is provably not a period. For instance, it would be a surprise to many
  if $1 / \pi$ were a period.
\end{remark}

\section{Products involving the Thue--Morse sequence}\label{sec:tm}

The Thue--Morse sequence $t_j$ is defined by $t_j = 1$ if the number of ones
in the binary representation of $j$ is odd and $t_j = 0$ otherwise. For
further information on this sequence and its occurrences in various contexts a
beautiful reference is {\cite{as-thuemorse}}. Let $p (j) = (- 1)^{t_j}$. Then
Theorem \ref{thm:prodrat} implies that, for $m > 1$,
\begin{equation}
  \prod_{k \geqslant 1} \prod_{j = 0}^{2^m - 1} (k + j)^{p (j)} = \prod_{j =
  0}^{2^m - 1} (j!)^{- p (j)} . \label{eq:prodtm}
\end{equation}
As we will see below, the right-hand side of (\ref{eq:prodtm}) further
simplifies.

\begin{remark}
  Note that $p (j) = 1$ for exactly half of the $2^m$ many numbers $j = 0, 1,
  \ldots, 2^m - 1$, so that the left-hand side of (\ref{eq:prodtm}) indeed
  converges when $m > 1$. Denote with $S_1$ the set of integers among $0, 1,
  \ldots, 2^m - 1$ with $p (j) = 1$. Similarly, $S_{- 1}$ consists of those
  with $p (j) = - 1$. Not only are the two sets equinumerous, but also
  $\sum_{j \in S_1} j = \sum_{j \in S_{- 1}} j$ and, in fact, as discovered by
  Prouhet in 1851, see {\cite{wright-prouhet}} or {\cite{as-thuemorse}},
  \[ \sum_{j \in S_1} j^n = \sum_{j \in S_{- 1}} j^n \]
  for all $n = 0, 1, \ldots, m - 1$.
\end{remark}

\begin{example}
  \label{eg:tm3}For instance, when $m = 3$ then $0^n + 3^n + 5^n + 6^n = 1^n +
  2^n + 4^n + 7^n$ for $n = 0, 1, 2$ as well as
  \[ \prod_{k \geqslant 1} \frac{k (k + 3) (k + 5) (k + 6)}{(k + 1) (k + 2) (k
     + 4) (k + 7)} = \frac{2!4!7!}{3!5!6!} = \frac{7}{3 \cdot 5} . \]
\end{example}

In the next lemma, we observe that the right-hand side of (\ref{eq:prodtm})
always simplifies as it did in Example \ref{eg:tm3}.

\begin{lemma}
  \label{lem:prodtms}For integers $m > 1$,
  \begin{equation}
    \prod_{k \geqslant 1} \prod_{j = 0}^{2^m - 1} (k + j)^{p (j)} = \prod_{j =
    0}^{2^{m - 1} - 1} (2 j + 1)^{p (j)} . \label{eq:prodtms}
  \end{equation}
\end{lemma}

\begin{proof}
  Note that $p (2 j)$ and $p (2 j + 1)$ are always of opposite sign. Thus,
  \[ \prod_{j = 0}^{2^m - 1} (j!)^{- p (j)} = \prod_{j = 0}^{2^{m - 1} - 1}
     \left[ \frac{(2 j) !}{(2 j + 1) !} \right]^{p (2 j + 1)} . \]
  It only remains to use that $p (2 j + 1) = - p (j)$, which follows from the
  definition of the Thue--Morse sequence.
\end{proof}

\begin{example}
  Combining Example \ref{eg:tm3} with the corresponding product for $m = 4$,
  one finds
  \[ \prod_{k \geqslant 1} \frac{(k + 9) (k + 10) (k + 12) (k + 15)}{(k + 8)
     (k + 11) (k + 13) (k + 14)} = \frac{11 \cdot 13}{9 \cdot 15} . \]
\end{example}

From computing a few more instances, we are led to observe that both sides of
(\ref{eq:prodtms}) appear to approach $1 / 2$ as $m \rightarrow \infty$. In
other words,
\begin{equation}
  \lim_{m \rightarrow \infty} \prod_{j = 0}^{2^m - 1} (2 j + 1)^{p (j)} =
  \frac{1}{2} . \label{eq:tm-half}
\end{equation}
To see that this is indeed true, it is of advantage to also consider
\begin{equation}
  f_m (x) = \prod_{j = 0}^{2^m - 1} (x + j)^{p (j)} . \label{eq:tm-fm}
\end{equation}
This product arises in
\[ \prod_{k \geqslant 0} \prod_{j = 0}^{2^m - 1} (x + k + j)^{p (j)} =
   \prod_{j = 0}^{2^{m - 1} - 1} (x + 2 j)^{p (j)} = f_m ( \tfrac{1}{2}), \]
which is a natural extension of (\ref{eq:prodtms}) and can be proved in the
same way.

Now the truth of (\ref{eq:tm-half}) can be seen from the solution
{\cite{robbins-2692}} to the problem {\cite{woods-2692}} proposed by Woods.
Indeed, denoting with $f (x)$ the limit of $f_m (x)$ as $m \rightarrow
\infty$, we need to show that $f (1 / 2) = 1 / 2$. Note that
\[ f_m (x + \tfrac{1}{2}) = \prod_{j = 0}^{2^m - 1} (x + \tfrac{1}{2} + j)^{p
   (j)} = \prod_{j = 0}^{2^m - 1} (2 x + 2 j + 1)^{- p (2 j + 1)} . \]
Proceeding as in {\cite{robbins-2692}}, we thus find that
\begin{equation}
  f_{m + 1} (2 x) f_m (x + \tfrac{1}{2}) = f_m (x) . \label{eq:tm-fm-dup}
\end{equation}
With appropriate care one may now apply L'H\^ospital's rule to obtain
\[ f ( \tfrac{1}{2}) = \lim_{x \rightarrow 0} \frac{f (x)}{f (2 x)} = \frac{f'
   (0)}{2 f' (0)} = \frac{1}{2}, \]
as desired.

\begin{remark}
  Combining terms as in the proof of Lemma \ref{lem:prodtms}, we note that
  \begin{equation}
    f (x) = \lim_{m \rightarrow \infty} \prod_{j = 0}^{2^m - 1} (x + j)^{p
    (j)} = \prod_{j = 0}^{\infty} \left( \frac{2 j + x}{2 j + x + 1}
    \right)^{p (j)}, \label{eq:tm-f}
  \end{equation}
  because the infinite product converges. It further follows from
  (\ref{eq:tm-fm-dup}) that $f (1)^2 = f (1 / 2)$ and hence $f (1) = 1 /
  \sqrt{2}$. In light of (\ref{eq:tm-f}) this result is equivalently expressed
  as
  \begin{equation}
    P = \prod_{j = 0}^{\infty} \left( \frac{2 j + 1}{2 j + 2} \right)^{p (j)}
    = \frac{1}{\sqrt{2}},
  \end{equation}
  which is the evaluation asked for in problem {\cite{woods-2692}}. A
  beautiful alternative solution, avoiding analytic tools such as
  L'H\^ospital's rule, is given in {\cite{as-thuemorse}}. The clever
  alternative proof considers, besides $P$, the product
  \[ Q = \prod_{j = 1}^{\infty} \left( \frac{2 j}{2 j + 1} \right)^{p (j)}, \]
  and shows that
  \begin{eqnarray*}
    P Q & = & \frac{1}{2} \prod_{j = 1}^{\infty} \left( \frac{j}{j + 1}
    \right)^{p (j)}\\
    & = & \frac{1}{2} \prod_{j = 1}^{\infty} \left( \frac{2 j + 1}{2 j + 2}
    \right)^{p (2 j + 1)} \prod_{j = 1}^{\infty} \left( \frac{2 j}{2 j + 1}
    \right)^{p (2 j)}\\
    & = & \frac{1}{2} P^{- 1} Q.
  \end{eqnarray*}
  Cancelling $Q$, one has again derived $P = 1 / \sqrt{2}$. On the other hand,
  the quantity $Q$ is much more mysterious. It is not even known whether $Q$
  is irrational, let alone transcendental. Jeffrey Shallit has offered \$25
  for an answer to this question.
\end{remark}

\paragraph{Acknowledgements}
We wish to thank Pieter Moree for an interesting
and useful discussion on infinite products in general, and Bruce Berndt for
helpful comments and advice. Finally, the second author would like to thank
the Max-Planck-Institute for Mathematics in Bonn, where part of this work was
completed, for providing wonderful working conditions.


\end{document}